\documentclass[12pt]{article}%
\usepackage{graphicx}
\usepackage{mathrsfs}
\usepackage{amsmath}
\usepackage{amsfonts}
\usepackage[T2A]{fontenc}
\usepackage[cp1251]{inputenc}
\usepackage[russian,english]{babel}
\usepackage{amssymb}%

\usepackage{wasysym}
\newtheorem{theorem}{Theorem}
\newtheorem{acknowledgment}[theorem]{Acknowledgment}

\newtheorem{corollary}[theorem]{Corollary}

\newtheorem{proposition}[theorem]{Proposition}
\newtheorem{remark}[theorem]{Remark}

\newenvironment{proof}[1][Proof]{\textbf{#1.} }{\ \rule{0.5em}{0.5em}}
\usepackage{setspace}
\begin{document}
\author {Martial Longla\\ University of Cincinnati\\Department of mathematical sciences\\PO Box
210025, Cincinnati, Oh 45221-0025, USA.\\martiala@mail.uc.edu}
\title{Remarks on the speed of convergence of mixing coefficients and applications}
\maketitle

\abstract{In this paper, we study dependence coefficients for copula-based Markov chains. We provide new tools to check the convergence rates of mixing coefficients of copula-based Markov chains. We study Markov chains generated by the Metropolis-hastings algorithm and give conditions on the proposal that ensure exponential $\rho$-mixing, $\beta$-mixing and $\phi$-mixing. A general necessary condition on symmetric copulas to generate exponential $\rho$-mixing  or $\phi$-mixing is given.  At the end of the paper, we comment and improve some of our previous results on mixtures of copulas. 
}



\section{Introduction}
This work is motivated by questions raised after reading Chen and al. (2009) \cite{Wei} and Beare (2010) \cite{Beare}. Chen and al.(2009) \cite{Wei} have shown that Markov chains generated by the Clayton, Gumbel or Student copulas are geometrically ergodic. They used in their paper quantile transformations and small sets to show geometric ergodicity, but could not handle for instance the mixture of these copulas. In a recent paper, Longla and Peligrad (2012) \cite{Martial1} have shown that these examples are actually exponential $\rho$- mixing. We have also answered the open question on geometric ergodicity of convex combinations of geometrically ergodic reversible Markov chains. 

Quantifying the dependence among two or more random variables has been an enduring task for statisticians. Copulas are full measures of dependence among components of random vectors. Unlike marginal and joint distributions, which are directly observable, a copula is a hidden dependence structure that couples a joint distribution with its marginals. An early statistical application of copulas was given by Clayton (1978) \cite{Clayton}, where the dependence between two survival times in a multiple events study is modeled by the so-called Clayton copula
\begin{equation*}
C(x,y)=(x^{-\alpha}+y^{-\alpha}-1)^{-1/\alpha} \hskip3mm \alpha\geq 0.
\end{equation*}
The literature on copulas is growing fast. An excellent overview, guide to the literature and applications belongs to Embrechts and al. (2003) \cite{Embrechts}. In later research into copulas, a driving force has been in financial risk management, where they are used to model dependence among different assets in a portfolio.  Nelsen's monograph (2006) \cite{Nelsen} can be regarded as one of the best books for an introduction to copulas.
 
\subsection{Definitions}
\subsubsection{2-Copulas}{A 2-copula is a function $C:[0,1]\times [0, 1] \rightarrow [0, 1]=I$ for which $C(0, y)=C(x, 0)=0$ ($C$ is grounded), $C(1, x)= C(x, 1) = x $(each coordinate is uniform on $I$) and for all $[x_{1}, x_{2}]x[y_{1}, y_{2}]\subset I^{2}$, $C(x_{1}, y_{1})+C(x_{2}, y_{2})-C(x_{1}, y_{2})-C(x_{2}, y_{1})\geq 0.$ }
The definition doesn't mention probability, but in fact these conditions imply that $C$ is the joint cumulative distribution function of two random variables with margins uniform on $I$. It follows from the definition (see \cite[chapter 1]{Nelsen}) that the function $C$ is non-decreasing in each of the parameters, has partial derivatives almost everywhere with values between $0$ and $1$. The partial derivative of $C(x,y)$ with respect to x is a non-decreasing function of $y$ and conversely. A convex combination of 2-copulas is a 2-copula.

If $X_{1}, X_{2}$ are random variables with joint distribution $F$ and marginal distributions $F_{1},F_{2} $, then the function $C(x,y)$ defined  by $C(F_{1}(x_{1}), F_{2}(x_{2}))=F(x_{1}, x_{2})$  is a copula. Moreover, if the random variables are continuous, then the copula is uniquely defined by the joint distribution and the marginal distributions by the formula $F(F^{-1}_{1}(x_{1}), F^{-1}_{2}(x_{2}))=C(x_{1}, x_{2}).$ This fact is known as Sklar's theorem. The implication of the Sklar’s Theorem is that, after standardizing the effects of marginals, the dependence among components of $X=(X_{1},X_{2})$ is fully described by the copula. Indeed, most conventional measures of dependence can be explicitly expressed in terms of the copula.
We will use in this paper the following conventional notation:   $||g||^{2}_{2}=\int_{I}g^{2}(x)dx$, $A_{,i}(x_{1},x_{2})=\frac{\partial A(x_{1},x_{2})}{\partial x_{i}}$, $c(x,y)$ will be used for the density of $C(x,y)$, and $\mathcal{R}$ will be used for the Borel $\sigma$-algebra of $I$.

\subsubsection{Copulas and Markov processes}
 Copulas have been shown to be a more flexible way to define a Markov process, as in the case when one suspects that the marginal distributions of the states are not related to the distribution of the initial state. Using copulas will allow changes in single marginal distributions, without having to change all other distributions in the chain.

A stationary Markov chain $(X_{n}, n\in\mathbb{Z})$ can be defined by a copula $(C(x,y))$ and a one dimensional marginal distribution. For stationary Markov chains with uniform marginals on $[0,1]$, the transition probabilities for all $n\in \mathbb{Z}$ are $P(X_{n}\in A|X_{n-1}=x)=C_{,1}(x,y)$ for sets $A=(-\infty, y]$ (for more details, see Darsaw and al. (1992) \cite[theorem 3.1]{Darsow}).  This relationship was used by Chen and al. (2009) \cite{Wei} to show that stationary Markov processes defined by the Clayton, Gumbel or Student copulas are geometrically ergodic. 

\subsubsection{Dependence coefficients}
Many dependence coefficients have been studied in the literature, such as $\alpha_{n}$, $\beta_{n}$, $\rho_{n}$, $\phi_{n}$ among others. In this paper, we will mainly use the last 3 coefficients defined as follows. 

Given $\sigma$-fields $\mathscr{A},\mathscr{B}$:
$$\beta(\mathscr{A},\mathscr{B})=\mathbb{E}\sup_{B\in \mathscr{B}}|P(B|\mathscr{A})-P(B)|$$  
$$\rho(\mathscr{A},\mathscr{B})=\sup_{f\in L^{2}(\mathscr{A}),g\in L^{2}(\mathscr{B})}corr(f,g),$$
$$\phi(\mathscr{A},\mathscr{B})=\sup_{B\in \mathscr{B}, A\in \mathscr{A}, P(A)>0}|P(B|A)-P(B)|$$
Given the alternative form of the transition probabilities for a Markov chain generated by a copula and a marginal distribution with strictly positive density, it was shown in Longla and Peligrad (2012) \cite{Martial1} that these coefficients have the following simple form when the copula is absolutely continuous with density $c$, $\mathscr{A}=\sigma(X_{i}, i\leq 0)$ and $\mathscr{B}=\sigma(X_{i}, i\ge n)$:

$$ \beta_{n}=\int_{0}^{1}\sup_{B\in \mathcal{R}}|\int_{B}(c_{n}(x,y)-1)dy|dx,$$
$$\phi_{n}=\sup_{B\in \mathcal{R}}ess\sup_{x}|\int_{B}(c_{n}(x,y)-1)dy|,$$
$$\rho_{n}=\sup\{\int^{1}_{0}\int^{1}_{0}c_{n}(x,y)f(x)g(y)dxdy : ||g||_{2}=||f||_{2}=1, \mathbb{E}(f)=\mathbb{E}(g)=0\}.$$
Here, $c_n$ is the density of the random vector $(X_{0}, X_{n})$. In general the following inequalities hold (see Bradley (2007) \cite[Theorem 7.4, 7.5]{Brad} for more.):
\begin{eqnarray}
\beta_{n}\leq \phi_{n}, \quad 
 \rho_{n} \leq 2\sqrt{\phi_{n}}, \quad \rho_{n} \leq (\rho_{1})^{n}.
\end{eqnarray}
These coefficients are defined to assess the dependence structure of the Markov process, and provide necessary conditions for CLT and functional CLT and their rates of convergence. Some examples can be found in Peligrad (1997) \cite{Magda}, Peligrad (1993) \cite{Magda2} and the references therein. A stochastic process is said to be $\alpha$-mixing if $\alpha_{n}\rightarrow 0$; $\beta$- mixing if $\beta_{n}\rightarrow 0$ or $\rho$-mixing if $\rho_{n}\rightarrow 0$. The process is exponentially mixing, if the convergence rate is exponential.
A stochastic process is said to be geometrically ergodic, if for some $a\in (0,1)$ , and $n \in \mathbb{N},$ 
\begin{equation*}
\sup_{B\in \mathcal{R}}|P(X_{n}\in B|X_{0}=x)-P(X_{n}\in B)|\leq a^{n}W(x), \hskip3mm \mbox{with} \hskip2mm \mathbb{E}(W(X_{0}))< \infty.
\end{equation*}
By Theorem 2.1 in Nummelin and Tuominen (1982) \cite{Nemm}, {\it geometric ergodicity is equivalent to exponential $\beta$-mixing.} So, whenever we encounter the term, we may understand either of the two concepts. The main results on this topic can be found in Darsow and al. (1992) \cite{Darsow}, de la Pe\~{n}a and Gin\'{e} (1999) \cite{DG}, Ibragimov and Lentzas (2009) \cite{IL}. They characterized copula-based Markov chains. Joe (1997) \cite{Joe} proposed a class of parametric (strictly) stationary Markov models based on parametric copulas and parametric invariant distributions. Their setup of the problem was modified by Chen and Fan (2006) \cite{Chen-Fan}, who studied a class of semi-parametric stationary Markov models based on parametric copulas and nonparametric invariant distributions and analyzed the strength of dependence in the Markov chain. They showed that the temporal dependence measure is fully determined by the properties of the copulas.  Studying the strength of the dependence, Beare (2010) \cite{Beare} provided simple sufficient conditions for geometric $\beta $-mixing for copula-based Markov chains. He also provided the proof based on simulation, that the Clayton copula provides exponential $\rho$-mixing. The theoretical proof of this fact and many others was given by Longla and Peligrad (2012) \cite{Martial1}. 

In this paper we derive the copula of the Metropolis-hastings algorithm for a general proposal, and use one of our previous results to provide a condition for exponential $\rho$-mixing (or say a way to choose the proposal to guarantee exponential rates of convergence). We also have some remarks on the rate of convergence of Markov chains generated by convex combinations of copulas. These results are applied to the Frechet and Mardia families of copulas.

The paper is structured as follows: In section two below we provide new results on exponential $\rho$-mixing, and exponential $\beta$-mixing for some families of Metropolis-Hastings algorithms. In section 3 we comment and improve some of our results from Longla and Peligrad (2012) \cite{Martial1} on convex combinations of copulas and apply them to some families of copulas.

\section{Sufficient conditions for $\rho$-mixing}
 We shall recall that exponential $\rho$-mixing is equivalent to $\rho_{1}<1$ for Markov chains. This fact will be used in this section to improve some known results. 
\subsection{A new bound on $\rho_{1}$}
Combining Proposition 2 of Longla and Peligrad (2012) \cite{Martial1} and Theorem 3.1 of Beare (2010) \cite{Beare}, we have the following.

\begin{theorem} \label{Theo1}
\quad

A copula-based Markov chain generated by a symmetric copula with square integrable density is geometric $\beta$-mixing if it is geometric $\rho$-mixing. Moreover, if the density is strictly positive on a set of measure 1, then $\rho$-mixing also follows from geometric ergodicity.
\end{theorem}
\bigskip
Beare (2010) \cite{Beare} proved that for a copula with density bounded away from zero we have exponential $\rho$-mixing. These conditions imply actually a stronger mixing condition ($\phi$) as shown by Longla and Peligrad (2012) \cite{Martial1}. These conditions are relaxed in the following theorem.
\bigskip
\begin{theorem} \label{Theo2}
\quad

If the copula of the Markov process is such that there exists nonnegative functions $\varepsilon_{1},$ $\varepsilon_{2}$ defined on $[0,1]$ for which the density of the absolute continuous part of the copula denoted $c(x,y)$ satisfies the inequality
$$c(x,y)\geq \varepsilon_{1}(x)+\varepsilon_{2}(y)$$ with $\varepsilon_{1}$, $\varepsilon_{2} \in \mathbb{L}_{1}[0,1]$ such that at least one of the two functions has a non-zero integral, then the process is exponential $\rho$-mixing. \end{theorem}
\bigskip

\begin{proof}
  Let $f$, $g$ be two functions with $||f||_{2}=||g||_{2}=1, \mathbb{E}(f(X))=\mathbb{E}(g(X))=0$. We have 

\begin{equation}
2f(x)g(y)=f^{2}(x)+g^{2}(y)-(f(x)-g(y))^{2}. \label{Al}
\end{equation}

Therefore,

$$2\int_{I^2}f(x)g(y)C(dx,dy)=\int_{I^2}f^{2}(x)C(dx,dy)+\int_{I^2}g^{2}(y)C(dx,dy)-\int_{I^2}(f(x)-g(y))^{2}C(dx,dy).$$
Using the fact that $\int_{I}C(dx,dy)=dx$ and $\int_{I}f^{2}(x)=\int_{I}g^{2}(x)dx=1$, we obtain
$$ \int_{I^2}f^{2}(x)C(dx,dy)=\int_{I}f^{2}(x)\int_{I}C(dx,dy)=\int_{I}f^{2}(x)dx=1=\int_{I^2}g^{2}(y)C(dx,dy).$$ 
On the other hand, using $c(x,y)\geq \varepsilon_{1}(x)+\varepsilon_{2}(y)$ on a set of Lebesgue measure 1,
$$\int_{I^2}(f(x)-g(y))^{2}C(dx,dy)\geq \int_{I^2}(f(x)-g(y))^{2}(\varepsilon_{1}(x)+\varepsilon_{2}(y))dxdy=$$
$$=\int_{I^2}(f^{2}(x)+g^{2}(y)-2f(x)g(y))(\varepsilon_{1}(x)+\varepsilon_{2}(y))dxdy= I_{a}+I_{b},$$
where
$$I_{b}=\int_{I^2}f^{2}(x)\varepsilon_{2}(y)dxdy+\int_{I^2}g^{2}(y)\varepsilon_{2}(y)dxdy-2\int_{I^2}f(x)g(y)\varepsilon_{2}(y)dxdy,
$$
$$I_{a}=\int_{I}f^{2}(x)\varepsilon_{1}(x)dx\int_{I}dy+\int_{I}\varepsilon_{1}(x)dx\int_{I}g^{2}(y)dy
-2\int_{I}f(x)\varepsilon_{1}(x)dx\int_{I}g(y)dy.$$
The cross term is zero, $\int_{I}g^{2}(x)dx=1$ and $\int_{I}f^{2}(x)\varepsilon_{1}(x)dx\geq 0$, whence
$$I_{a}= \int_{I}f^{2}(x)\varepsilon_{1}(x)dx+\int_{I}\varepsilon_{1}(x)dx\geq \int_{I}\varepsilon_{1}(x)dx.$$
Similarly, $I_{b}\geq \int_{I}\varepsilon_{2}(y)dy$. Thus, 
$$-\int_{I^2}(f(x)-g(y))^{2}C(dx,dy)\leq -\int_{I}\varepsilon_{1}(x)dx-\int_{I}\varepsilon_{2}(y)dy.$$
Using formula (\ref{Al}) and integrating, we obtain
$$ 2\int_{I^2}f(x)g(y)C(dx,dy)\leq 2-(\int_{I}\varepsilon_{1}(x)dx+\int_{I}\varepsilon_{2}(y)dy).$$

Recalling the definition of correlation, it follows that 
\begin{equation}
corr(f,g)\leq 1-\frac{1}{2}(\int_{I}\varepsilon_{1}(x)dx+\int_{I}\varepsilon_{2}(y)dy). \nonumber
\end{equation}
Because this holds for all such $f$ and $g$, it holds for $f$ and $-g$. Thus 

\begin{equation}
\sup_{f,g}|corr(f,g)|\leq 1-\frac{1}{2}(\int_{I}\varepsilon_{1}(x)dx+\int_{I}\varepsilon_{2}(y)dy).
\end{equation} 
Provided that one of the two integrals of the right hand side is non-zero (say $\varepsilon$). It follows that 
\begin{equation}
\rho_{1} \leq 1-\frac{1}{2}\varepsilon < 1.
\end{equation} 
 Combining this inequality with Theorem \ref{Theo1}, we obtain exponential $\rho$-mixing.
\end{proof}

Combining Theorem \ref{Theo2} with the second part of Theorem \ref{Theo1} and the comment after Theorem 4 in Longla and Peligrad (2012) \cite{Martial1}, we obtain the following corollary:

\begin{corollary} \label{corrMH}
Under the assumptions of Theorem \ref{Theo2}, if one of the following conditions holds.
\begin{enumerate}
\item The density of the copula is strictly positive on a set of Lebesgue measure 1;
\item The density is symmetric and square integrable;
\end{enumerate}
then the copula generates geometrically ergodic Markov chains.
\end{corollary}
\bigskip

\subsection{Introduction to the Metropolis Hastings algorithm}

The Metropolis-Hastings algorithm is one of the tools that statisticians use to simulate data from not completely known distributions> They are used especially when the distributions are known up to the normalizing constant or when it is impossible to have independent observations from the given distributions. The investigator therefore looks for a Markov chain, whose transition probabilities converge to an invariant distribution equivalent to the necessary incomplete one. A sample from the stationary phase of this Markov chain is then used to study questions related to the initial unknown density. In order to achieve this goal, one should know the convergence rates of the dependence coefficients of the Markov chain. The rates of convergence provide accuracy of the approximation of the start of the stationary phase of the chain. The mixing rates are important also for inference using statistics of the simulated data (they provide corresponding limit theorems for inference and testing). All these properties are strongly related to the choice of the proposal for the simulation process. We will derive here the formula of the copula of these  processes for some classes of proposals, that one can use along with other theorems in this paper to choose a proposal that provides necessary rates of convergence of the transition probabilities.

Let $f$ be the invariant density we want to generate observations from using the Metropolis-Hastings algorithm. Assume $f$ is strictly positive on its domain with cumulative distribution function $F$. Let $q(x,y)$ be the proposal defined and positive on the domain of $f$ (q is a density for each fixed value of $x$).
 Define, 
$$\alpha(x,y)=\min{\{\frac{f(y)q(y,x)}{f(x)q(x,y)}, 1\}}.$$
Start the chain by a value in the domain of the distribution of interest, then sample potential new states from the proposal distribution $q$ which must be easy to sample from, and depends on the current state of the chain. More precisely, if the chain is initialized at $X_{0} = x_{0}$, then at any time $t\ge 1$ the Metropolis-Hastings algorithm performs the following steps:
\begin{enumerate}
\item Generate an observation $y$ from  q, using $x_{t-1}$ as parameter $x$,
\item Compute the acceptance ratio $\alpha(x,y)$ defined above,
\item With probability $\alpha(x,y)$, set $X_{t} = y$ and otherwise $X_{t} = x_{t-1}.$
\end{enumerate}
Implementation of step 3 requires generating an observation $u$ from the uniform distribution on $[0,1]$, then setting  $X_{t} = y$ if $\alpha(x,y)\leq u$, otherwise $X_{t} = x_{t-1}.$ It is therefore easy to derive the transition kernel of the obtained Markov chain:
\begin{equation}
P(x,dy)=q(x,y)\alpha(x,y)dy+(1-\int{q(x,y)\alpha(x,y)}dy)\delta_{x}(dy).
\end{equation}

The independent Metropolis-Hastings algorithm is obtained by choosing an independent proposal (which is independent of the current state $q(y)$). A list of applied examples of this algorithm can be found in Craiu (2011) \cite{Radu}. To simplify the formula of $\alpha$, statisticians often use a symmetric proposal. For a general proposal $q(x,y)$, the following holds.

\begin{proposition}
The copula of the Metropolis-Hastings algorithm with a proposal $q(x,y)$ and an invariant density $f$ with cumulative distribution $F$ is given by
$$C(u,v)=AC(u,v)+SC(u,v),$$
$$AC(u,v)=\int_{0}^{u}\int_{0}^{v}\frac{q(F^{-1}(t),F^{-1}(z))}{f(F^{-1}(z))}\alpha(F^{-1}(t),F^{-1}(z))dzdt$$
$$SC(u,v)=\int_{0}^{\min\{u,v\}}\big[ 1-\int_{0}^{1}\frac{q(F^{-1}(t),F^{-1}(z))}{f(F^{-1}(z))}\alpha(F^{-1}(t),F^{-1}(z))dz\big]dt
.$$ 
\end{proposition}
\begin{proof}
Using the definition of the copula from the Sklar's Theorem and the relationship with the transition probability given in the Introduction, we obtain
$$C_{,1}(F(x),F(y))=P(x, (-\infty,y])=\int_{-\infty}^{y}P(x,dt).$$
Thus, using the chosen proposal and a simple change of variable, we obtain
$$C_{,1}(u,v)=\int_{-\infty}^{F^{-1}(v)}\alpha(F^{-1}(u),t)q(F^{-1}(u),t)dt+
(1-\int_{-\infty}^{\infty}\alpha(F^{-1}(u),t)q(F^{-1}(u),t)dt)\mathbb{I}(u\leq v)$$

The change of variable inside the integral ($z=F(t)$) and integration with respect to $u$ combined with uniqueness of the copula lead to:

\begin{eqnarray}
AC(u,v)=\int_{0}^{u}\int_{0}^{v}\frac{q(F^{-1}(t),F^{-1}(z))}{f(F^{-1}(z))}\alpha(F^{-1}(t),F^{-1}(z))dzdt \label{z1}\\
SC(u,v)=\int_{0}^{\min\{u,v\}}\big[ 1-\int_{0}^{1}\frac{q(F^{-1}(t),F^{-1}(z))}{f(F^{-1}(z))}\alpha(F^{-1}(t),F^{-1}(z))dz\big]dt
. \label{z2}\end{eqnarray} 
\end{proof}

$AC$ is called the absolute continuous part of $C$ and $SC$ is the singular part of $C$. 

For a strictly positive symmetric proposal $q(x,y)$, we note that $\alpha$ simplifies; (\ref{z1}) and (\ref{z2}) become
\begin{eqnarray}
AC(u,v)=\int_{0}^{u}\int_{0}^{v}\frac{q(F^{-1}(t),F^{-1}(z))}{f(F^{-1}(z))}\min\{\frac{f(F^{-1}(z))}{f(F^{-1}(t))},1\}dzdt \label{zx1}\\
SC(u,v)=\int_{0}^{\min\{u,v\}}\big[ 1-\int_{0}^{1}\frac{q(F^{-1}(t),F^{-1}(z))}{f(F^{-1}(z))}\min\{\frac{f(F^{-1}(z))}{f(F^{-1}(t))},1\}dz\big]dt \label{zx2}
\end{eqnarray} 

\subsection{Mixing rates of the Metropolis-Hastings based on the proposal}

The choice of the proposal is crucial in the convergence of the Metropolis-Hastings algorithm. Here, we provide some conditions to measure ahead of time the rate of convergence for various classes of Metropolis-Hastings algorithms.

\begin{corollary}
The Independent Metropolis-Hastings algorithm generates a geometrically ergodic and exponential $\rho$-mixing Markov chain if for almost all $y$, $q(y)\ge a f(y)$, for some $a > 0$.
\end{corollary}

A related theorem was proved by Mengersen and Tweedie (1996)\cite{Mengersen}. We provide here a proof of this statement based on Theorem \ref{Theo2}.

\begin{proof}
The density of the absolute continuous part of the copula of the Metropolis-Hastings algorithm is 
$$c(u,v)=\frac{q(F^{-1}(u),F^{-1}(v))}{f(F^{-1}(v))}\alpha(F^{-1}(u),F^{-1}(v)),$$
For the independence Metropolis-Hastings algorithm we obtain
$$c(u, v)\ge \frac{q(F^{-1}(v))}{f(F^{-1}(v))}\min\{a\frac{f(F^{-1}(v))}{q(F^{-1}(v))},1\}\ge a \min\{a\frac{f(F^{-1}(v))}{q(F^{-1}(v))},1\}.$$
Therefore, by Theorem \ref{Theo2}, the chain is exponential $\rho$ mixing. This chain being also reversible and absolutely regular (positive density of absolute continuous part of the copula on a set of measure 1), we can conclude that the chain is geometrically ergodic by Corollary \ref{corrMH}. 
\end{proof}

A similar theorem holds for a non-symmetric proposal $q(x,y)$.

\begin{theorem} \label{Theo6}
\quad

Assume that the proposal $q(x,y)$ is bounded from above on its support, the support is continuous and $q(x,y)\ge a f(y)$, then the Metropolis-Hastings algorithm generates an exponential $\rho$-mixing Markov chain. Moreover, if $f$ is strictly positive, then the generated Markov chain is geometrically ergodic, and if $f$ is bounded away from zero on a set of Lebesgue measure 1, then the generated Markov chain is uniformly mixing.
\end{theorem}
\bigskip
 \bigskip

\begin{proof}
The density of the absolute continuous part of the copula of the Metropolis-Hastings algorithm is 
$$c(u,v)=\frac{q(F^{-1}(u),F^{-1}(v))}{f(F^{-1}(v))}\alpha(F^{-1}(u),F^{-1}(v)).$$
This is clearly a symmetric function. Thus, the metropolis-Hastings algorithm generates reversible Markov chains.  Given the assumptions of Theorem \ref{Theo6}, we have
$$c(u, v)\geq \frac{q(F^{-1}(u),F^{-1}(v))}{f(F^{-1}(v))}\min\{a\frac{f(F^{-1}(v))q(F^{-1}(v),F^{-1}(u))}{q(F^{-1}(u),F^{-1}(v))f(F^{-1}(u))},1\}\ge a \min\{a\frac{f(F^{-1}(v))}{q(F^{-1}(u),F^{-1}(v))},1\}.$$
Now, we use $q(x,y)<k$ for almost all $(x,y)$ as $q(x,y)$ is bounded from above on its domain. It follows that
$$c(u,v)\geq a \min\{a\frac{f(F^{-1}(v))}{k},1\}.$$ 
Therefore, by Theorem \ref{Theo2}, the chain is exponential $\rho$-mixing. As for the second part of the theorem, f is strictly positive on a set of measure 1 implies that the density of the absolute continuous part of the copula is positive on a set of measure 1. Thus, by Proposition 2 in Longla and Peligrad (2012) \cite{Martial1}, the Markov chain is absolutely regular. We can conclude that the chain is geometrically ergodic by Theorem 4 in Longla and Peligrad (2012) \cite{Martial1}. For $f$ bounded away from 0,  we obtain a density bounded away from zero, and by Theorem 8 in Longla and Peligrad (2012) \cite{Martial1}, the generated markov chain is $\phi$-mixing.
\end{proof}

\begin{remark}
Mengersen and Tweedie (1996) \cite{Mengersen} have shown that the fact that without this bound on the ratio in the case of the independent algorithm the chain may tend to "stick" in regions of low density is of considerable practical importance and is not merely a curiosity: seemingly sensible procedures give this behavior. They provided an example to illustrated this fact.
\end{remark}

\section{On mixing rates of mixtures copula-Based Markov chains}
\subsection{Convex combinations or mixtures of copulas}

Mixing coefficients for copula-based Markov chains have been widely studied recently. Here are some implications of the results in Longla and Peligrad (2012) \cite{Martial1}.

\begin{theorem}
\quad

Let $(C_{k}(x,y)$; $1\leq k\leq n)$ be $n$ copulas.
\begin{enumerate}

\item If $C_{1}(x,y)$ generates exponential $\rho$-mixing, then any convex combination of these copulas that contains $C_{1}(x,y)$ generates exponential $\rho$-mixing stationary Markov chains.
\item If $C_{2}$(x,y) generates an absolutely regular stationary Markov chain, and $C_{1}(x,y)$ generates an exponential $\rho$-mixing stationary Markov chain, then any convex combination containing $C_{1}(x,y)$ and $C_{2}(x,y)$ generates geometrically ergodic and exponential $\rho$-mixing Markov chains. 
\end{enumerate}
\end{theorem}

\begin{proof}
The proof of the first part of the conclusion is similar to the proof of Theorem 5 of Longla and Peligrad (2012) \cite{Martial1}. We shall note that symmetry of the copulas was not needed there. Symmetry in the proof of Theorem 5 in Longla and Peligrad (2012)\cite{Martial1} was not actually necessary for the conclusion on geometric ergodicity. The proof of the second part does not need the reversibility condition as well because what matters is that the Markov chain that is generated by the convex combination is absolutely regular and $\rho$-mixing by part one. So, all arguments of the proof of Theorem 5 in Longla and Peligrad (2012) \cite{Martial1} are valid.
\end{proof}

\begin{remark}
Basically, this theorem states that any convex combination of copulas that contains a copula that generates $\rho$-mixing will inherit this property.This happens no matter what are the mixing properties of other copulas in the combination. Moreover, if this convex combination generates absolutely regular Markov chains, then this combination generates geometrically ergodic Markov chains.
\end{remark}
This is an implication of this result that can be useful in applications.

\begin{corollary}
\label{convex}
\quad

Assume $(C_{k}(x,y)$; $1\leq k\leq n)$ are $n$ copulas
and $C_{1}(x,y)$ has the density of its absolute continuous part strictly
positive on a set of Lebesgue measure $1.$ Assume $C_{2}(x,y)$ generates a $\rho$-mixing stationary Markov chain. Then, any convex combination, $\sum_{k=1}^{n}a_{k}%
C_{k}(x,y)$ with $\sum_{k=1}^{n}a_{k}=1,$ $0\leq a_{i}\leq1$, $a_{1}\neq0$ and $a_{2}\neq0$
generates a geometrically ergodic Markov chain.
\end{corollary}
\bigskip
\subsection{Applications}
The Hoeffding lower and upper bounds are defined respectively as $W(x,y)=\max{(x+y-1,0)}$ and $M(x,y)=\min{(x,y)}$. The independence copula is defined by $P(x,y)=xy$.
\begin{enumerate}

\item The Mardia family of copulas 
\begin{equation}
C(x,y)=\frac{\theta^{2}(1+\theta)}{2}M(x,y)+(1-\theta^2)P(x,y)+\frac{\theta^{2}(1-\theta)}{2}W(x,y), \label{Mardia} \quad \theta\in[-1,1],
\end{equation}
 The formula (\ref{Mardia}) defines the Mardia family of copulas.  This family of copulas is a convex combination of the copulas $M$, $W$ and $P$.
\item The Frechet family of copulas 
\begin{equation}
C(x,y)=a M(x,y) +(1-a-b)P(x,y)+b W(x,y) \quad (0\leq a+b\leq 1). \label{Frechet}
\end{equation}
The formula (\ref{Frechet}) defines the Frechet family of copulas. Notice that a Mardia copula is a Frechet copula with $a+b=\theta^{2}$. So, they will share the same properties.
Given that the copula $P$ generates geometrically ergodic and $\rho$-mixing stationary Markov chains and these convex combinations are symmetric, we can conclude by the above Corollary \ref{convex} that these families generate exponential $\rho$-mixing and geometrically ergodic reversible Markov chains for $a+b\ne 1$ ($\theta^{2}\neq 1$) respectively.
\end{enumerate}
\begin{acknowledgment}
The Author thanks his advisor M. Peligrad for useful comments and all the help she provides in everything.
\end{acknowledgment}

\end{document}